\documentclass[12pt]{article}
\usepackage[utf8]{inputenc}
\usepackage[T1]{fontenc}
\usepackage[numbers,sort&compress]{natbib}
\usepackage[margin=1in]{geometry}
\usepackage{amsmath, amssymb, amsthm}
\usepackage{graphicx}
\usepackage{hyperref}
\usepackage{booktabs}
\usepackage{array}
\usepackage{xcolor}
\usepackage{setspace}
\usepackage{caption}
\usepackage{lineno}

\hypersetup{colorlinks=true,linkcolor=blue!60!black,citecolor=blue!60!black,urlcolor=blue!60!black}
\newtheorem{theorem}{Theorem}[section]
\newtheorem{proposition}[theorem]{Proposition}

\newtheorem{remark}{Remark}[section]

\title{What \texorpdfstring{$R_0$}{R0} Deletes:
Eigenvectors, Non-Normality, and the Social Content
of the Basic Reproduction Number}
\author{%
Fabio Sanchez\\[3pt]
  {\small Escuela de Matemática-CIMPA, Universidad de Costa Rica,}\\
  {\small Ciudad Universitaria Rodrigo Facio, San José, 11501, Costa Rica.}
}
\date{\today}

\begin{document}

\maketitle

\begin{abstract}
The basic reproduction number is the spectral radius of a matrix, $R_0=\rho(K)$. Taking that definition literally, we ask what $K\mapsto\rho(K)$ discards. A matrix carries three kinds of information: its dominant eigenvalue, its dominant eigenvectors, and its departure from normality. $R_0$ keeps only the first; the other two are where the epidemic's social structure lives. The right eigenvector is the burden distribution, the left the source distribution; they coincide when the system is normal and diverge under heterogeneity. Across the $177$ national contact matrices of Prem et al., the operator is \emph{never} normal, and once age-specific susceptibility is included, its source and burden eigenvectors are misaligned by a median of $26^{\circ}$, exceeding $40^{\circ}$ in some countries: the groups that drive transmission are systematically not those that bear it. We prove that under reciprocal contact this misalignment obeys a Kantorovich bound set by the susceptibility contrast $q_{\max}/q_{\min}$ alone, and zero when susceptibility is uniform, with the excess in real, non-reciprocal matrices contributed by contact asymmetry. Transient amplification, by contrast, stays small, so the operative social content is the misalignment, not transient blow-up. The omission also has teeth: because minimizing $R_0$ protects those who \emph{spread} infection, while minimizing deaths protects those who \emph{die} from it, the two target different age groups; the former sometimes raises average infection fatality even as it lowers the scalar. When contact is strongly structured and susceptibility is heterogeneous, we suggest reporting $R_0$ along with its eigenvectors rather than reporting it alone.
\end{abstract}


\section{Introduction}

The basic reproduction number is usually introduced as the expected number of secondary infections produced by one infectious individual in a fully susceptible population, with the threshold reading that infection can invade when $R_0>1$. In
structured models, this quantity is defined precisely as the spectral radius of the next-generation matrix $K$,
\begin{equation}
R_0=\rho(K),
\end{equation}
where $K_{ij}$ is the expected number of infections generated in group $i$ by one infected individual of group $j$ \cite{diekmann2010,vandendriessche2002}.

Much has been written about $R_0$ being ``context dependent'' rather than an intrinsic property of a pathogen \cite{delamater2019}. That claim is correct, but it is usually left as prose. It becomes sharp once we notice what kind of object $\rho(\cdot)$ is. A matrix encodes three separable pieces of information: its leading eigenvalue, its leading eigenvectors, and its departure from normality (the failure of its eigenvectors to be orthogonal). The map $K\mapsto\rho(K)$
retains only the first. Whatever social structure lives in the other two is deleted by construction, not by oversight.

The organizing idea of this paper is that those deleted pieces are the social content of the epidemic, each with a clean epidemiological meaning:
the leading \emph{eigenvector} is the distribution of harm; the \emph{non-normality} is the heterogeneity; and both are erased the moment $K$ is collapsed to $R_0$.
We first develop the argument mathematically (Sections~\ref{sec:nonnormal}--\ref{sec:transient}), proving in particular a Kantorovich bound (Proposition~\ref{prop:kant}) that caps the source--burden misalignment by the susceptibility contrast alone; we then test it on real national contact structures (Section~\ref{sec:empirical}), which both confirm the equity claim and temper the
transient one.


\paragraph{Contribution.}\emph{What is new here} is three concrete additions, none of them the interpretation itself. \emph{(i) A bound.} For reciprocal contacts, the source--burden misalignment obeys a Kantorovich inequality, controlled by the susceptibility contrast $q_{\max}/q_{\min}$ alone and vanishing under homogeneity (Proposition~\ref{prop:kant}), reducing the misalignment to a single scalar fixed by the susceptibility profile. \emph{(ii) A large-scale measurement.} Across the $177$ national contact structures of Prem et al.\ we quantify the misalignment, show it is never zero, and separate it into a susceptibility component that the bound caps and a contact-asymmetry remainder, characterizing next-generation-operator non-normality across every country for which synthetic contact data exist. \emph{(iii) A decision consequence.} We show the discarded eigenvector changes an intervention: $R_0$-optimal and mortality-optimal protection target different age groups in all $177$ countries.

\emph{What is not new} is the eigenstructure these rest on, and it is worth stating plainly. In matrix population biology the dominant right and left eigenvectors of a projection operator are the stable stage distribution and Fisher's reproductive value; their systematic development is due to Caswell \cite{caswell2001}. Epidemiology inherited both through the next-generation matrix \cite{diekmann2010,vandendriessche2002}: $R_0=\rho(K)$ is its dominant eigenvalue, the right eigenvector $v$ is the stable distribution of infections across groups (the case mix of an invading outbreak), and the left eigenvector $w$ is the group-specific reproductive value. That the two need not coincide is implicit in any asymmetric transmission model, and the sensitivity $\partial\rho/\partial K_{ij}\propto w_i v_j$ is the standard elasticity result underlying the theory of whom to vaccinate or shield \cite{caswell2001}. Non-normality and its transient consequences were formalized in numerical linear algebra \cite{trefethen2005} and imported into ecology as \emph{reactivity} by Neubert and Caswell \cite{neubert1997}, thence to epidemic transient growth \cite{hosack2008}; and that $R_0$ reflects host social structure rather than being pathogen-intrinsic is recognized \cite{delamater2019}. What this literature has generally not done is treat the \emph{gap} between $w$ and $v$ as an object in its own right, a measurable, boundable, interpretable quantity rather than as intermediate machinery for computing a threshold or a sensitivity. That gap is our subject.

We therefore claim no novelty for the interpretation. That the eigenvectors carry the distributions of burden and transmission, and that $R_0=\rho(K)$ suppresses them, follows from the sources above and is, in hindsight, immediate; the same is true of the observation that $R_0$ absorbs social structure. The incremental contribution is items (i)--(iii); the projection framing we use to organize them is expository, and a reader who finds the interpretation familiar loses nothing by treating it as a lens on those three results.

\section{The transmission matrix is generically non-normal}\label{sec:nonnormal}

It is useful to separate the mechanisms bundled into $K$ through the conceptual decomposition
\begin{equation}
K=Q\,C\,D,
\end{equation}
where $C$ is the social contact structure (who meets whom, and how often), $Q$ is a diagonal matrix of per-contact biological transmissibility and susceptibility, and $D$ is a diagonal matrix of time spent in transmission-relevant states. The
factorization is a device, not a theorem; its purpose is to locate where heterogeneity enters.

Contacts are often close to reciprocal, so one might take $C$ symmetric and hope this makes $K$ symmetric, hence \emph{normal}, so that eigenvalues alone would tell the whole story. They do not. A matrix is normal when $KK^{\top}=K^{\top}K$, equivalently when its eigenvectors can be chosen orthogonal. With $C$ symmetric,
\begin{equation}
KK^{\top}=QCD^{2}CQ,\qquad K^{\top}K=DCQ^{2}CD,
\end{equation}
and these agree only when $Q$ and $D$ are scalar multiples of the identity; that is, only when every group has identical transmissibility and identical infectious duration. \emph{Unequal biology or unequal exposure time across groups makes $K$ non-normal even when contacts are perfectly reciprocal.} Non-normality is not an exotic edge case here; it is the mathematical footprint of heterogeneity itself. (The eigenvalues may remain real, so this is not a statement about oscillation; the effects below occur in the plain metric of counted infections.)

\section{The eigenvector is the equity content}\label{sec:equity}

Iterating infection by generations, $x_{n+1}=Kx_{n}$, the population settles onto the leading right eigenvector $v$, defined by $Kv=\rho v$. Thus $v$ is the stable distribution of \emph{incident infections across groups}: who gets infected. The
leading left eigenvector $w$, with $w^{\top}K=\rho w^{\top}$, is the \emph{reproductive value}: how much a single case in each group contributes to long-run transmission \cite{caswell2001}. Reproductive value is a per-case quantity; a group's realized contribution to spread combines it with the group's incidence, $w_g v_g$ (this product reappears in Section~\ref{sec:decision} as the sensitivity of $R_0$ to protecting group $g$).

For a normal $K$ the eigenvectors are orthogonal and, in the symmetric case, $v=w$: the groups that drive transmission are the groups that bear it. Reciprocity of harm is exactly what normality encodes. Heterogeneity is the opposite statement: the source distribution $w$ and the burden distribution $v$ point in
different directions. Intuitively, susceptibility governs how readily a group is \emph{infected}, a sink property that raises its burden, not how much it goes on to \emph{infect others}, which under reciprocal contact and uniform infectiousness is set separately. Concentrating susceptibility in some groups therefore places a burden on them without a matching rise in their onward transmission, tilting $v$ toward the susceptible and $w$ away from them; only uniform susceptibility keeps the two aligned. A single scalar measures how far apart they are, the condition number of the leading eigenvalue,
\begin{equation}
\kappa=\frac{\lVert v\rVert\,\lVert w\rVert}{\lvert w^{\top}v\rvert}\;\ge\;1,
\end{equation}
with $\kappa=1$ precisely in the aligned, egalitarian case and $\kappa$ growing as sources and sufferers separate. This is the quantity $R_0$ cannot contain: two societies with identical $R_0$ can carry very different $\kappa$, one spreading burden reciprocally ($\kappa=1$, $\theta=0$) and the other concentrating it on a group that transmits little in return (in the real data of Section~\ref{sec:empirical}, up to $\kappa\approx1.37$, i.e.\ $\theta\approx43^{\circ}$; larger still if susceptibilities are more unequal). The eigenvalue indicates that invasion is possible; only the eigenvectors indicate to whom it happens. Section~\ref{sec:bound} bounds this misalignment in terms of the susceptibility contrast alone; Section~\ref{sec:empirical} measures it on real contact data and finds it substantial.

\begin{table}[h]
\centering\small
\begin{tabular}{@{}p{0.36\textwidth} p{0.58\textwidth}@{}}
\toprule
Quantity & Epidemiological reading \\
\midrule
$\rho(K)=R_0$ (dominant eigenvalue) & Invasion threshold, the one piece $R_0$ keeps. \\
$v$ (right eigenvector) & Burden: who becomes infected. \\
$w$ (left eigenvector) & Reproductive value: who drives transmission. \\
$\theta,\ \kappa=1/\cos\theta$ (source--burden angle, condition number) & How far $w$ and $v$ separate; $\theta=0$ when aligned. \\
$\nu$ (departure from normality) & How far $K$'s eigenvectors are from orthogonal. \\
$\sigma_{\max}(K)/\rho$ (reactivity ratio) & Largest transient surge a subcritical $K$ allows. \\
\bottomrule
\end{tabular}
\caption{Key quantities. $R_0$ retains only the first row; the rest are the ``social content'' the spectral radius discards, and all are computed en route to $\rho(K)$. Defined in Sections~\ref{sec:nonnormal}--\ref{sec:transient}.}
\label{tab:defs}
\end{table}

\section{How far can susceptibility bend the eigenvectors?}\label{sec:bound}

Section~\ref{sec:equity} identified the condition number $\kappa$ as the carrier of the equity content; we now bound it. Consider the on-model case $K=QC$ with $C$ symmetric (reciprocal contacts) and $Q=\operatorname{diag}(q_1,\dots,q_n)$, $q_i>0$, the age-specific susceptibility; this is the operator of Section~\ref{sec:empirical} with contact taken reciprocal. In plain terms, the result below says that when mixing is reciprocal, the source and burden distributions cannot be pulled arbitrarily far apart: the angle between them is capped by a simple function of a single number, the ratio of the most to the least susceptible group, and closes to zero when susceptibility is uniform.

\begin{proposition}\label{prop:kant}
Let $C$ be symmetric, entrywise nonnegative and irreducible, and $Q=\operatorname{diag}(q_i)$ with $q_i>0$. Write $K=QC$, with Perron eigenvalue $\rho$, positive right eigenvector $v$ $(Kv=\rho v)$ and left eigenvector $w$ $(w^{\top}K=\rho w^{\top})$. Then:
\begin{enumerate}
\item[\textup{(i)}] $K$ is diagonally symmetrizable, $K=Q^{1/2}SQ^{-1/2}$ with $S=Q^{1/2}CQ^{1/2}$ symmetric; its eigenvalues are real, and $v=Q^{1/2}\varphi$, $w=Q^{-1/2}\varphi$, where $\varphi>0$ is the unit Perron eigenvector of $S$ (its dominant eigenvector, with all entries positive).
\item[\textup{(ii)}] Writing $p_i=\varphi_i^{2}$ \textup{(}a probability vector over groups\textup{)} and $\mathbb{E}_p[\cdot]$ for the corresponding weighted average, the condition number and source--burden angle satisfy
\begin{equation}
\kappa=\frac{\lVert v\rVert\,\lVert w\rVert}{w^{\top}v}
=\Big(\textstyle\sum_i p_i q_i\Big)^{1/2}\Big(\sum_i p_i q_i^{-1}\Big)^{1/2}
=\big(\mathbb{E}_p[q]\,\mathbb{E}_p[q^{-1}]\big)^{1/2},
\qquad \cos\theta=\frac1\kappa .
\end{equation}
\item[\textup{(iii)}] With $m=\min_i q_i$, $M=\max_i q_i$ and contrast $r=M/m$,
\begin{equation}
1\;\le\;\kappa\;\le\;\tfrac12\!\left(\sqrt{r}+\tfrac{1}{\sqrt{r}}\right),
\qquad\text{equivalently}\qquad
\theta\;\le\;\arccos\frac{2\sqrt{r}}{1+r}.
\end{equation}
The lower bound holds with equality iff $Q$ is a scalar matrix. The upper bound is sharp: it is attained when $q$ takes only its extreme values $\{m,M\}$ (with $C$ placing equal Perron mass on them), and is a strict supremum otherwise, since a positive Perron vector must assign positive weight to any intermediate susceptibilities.
\end{enumerate}
\end{proposition}

\begin{proof}
\textup{(i)} $Q^{1/2}SQ^{-1/2}=Q^{1/2}\bigl(Q^{1/2}CQ^{1/2}\bigr)Q^{-1/2}=QC=K$, and $S^{\top}=Q^{1/2}C^{\top}Q^{1/2}=S$. Being similar to the symmetric matrix $S$, $K$ has real eigenvalues. Since $S$ is symmetric, nonnegative and irreducible, Perron--Frobenius provides a unit eigenvector $\varphi>0$ with $S\varphi=\rho\varphi$ and $\rho=\rho(K)$. Then $K\bigl(Q^{1/2}\varphi\bigr)=Q^{1/2}S\varphi=\rho\,Q^{1/2}\varphi$, and $\bigl(Q^{-1/2}\varphi\bigr)^{\top}K=\rho\bigl(Q^{-1/2}\varphi\bigr)^{\top}$.

\textup{(ii)} $w^{\top}v=\varphi^{\top}Q^{-1/2}Q^{1/2}\varphi=\varphi^{\top}\varphi=1$; moreover $\lVert v\rVert^{2}=\varphi^{\top}Q\varphi=\sum_i q_i\varphi_i^{2}$ and $\lVert w\rVert^{2}=\varphi^{\top}Q^{-1}\varphi=\sum_i q_i^{-1}\varphi_i^{2}$. Hence $\kappa=\lVert v\rVert\,\lVert w\rVert$ takes the stated form and $\cos\theta=w^{\top}v/(\lVert v\rVert\lVert w\rVert)=1/\kappa$.

\textup{(iii)} By Cauchy--Schwarz, $\mathbb{E}_p[q]\,\mathbb{E}_p[q^{-1}]\ge\bigl(\mathbb{E}_p[\sqrt q\cdot q^{-1/2}]\bigr)^{2}=1$, with equality iff $q$ is $p$-almost-surely constant; as $\varphi>0$ this forces all $q_i$ equal. The Kantorovich inequality, $\mathbb{E}_p[q]\,\mathbb{E}_p[q^{-1}]\le (M+m)^{2}/(4Mm)$ for $q_i\in[m,M]$, gives the upper bound after writing $r=M/m$ and taking square roots; the angle form follows from $\cos\theta=1/\kappa$.
\end{proof}

\begin{remark}
In words, $\kappa$ is a weighted-average susceptibility times a weighted-average \emph{inverse} susceptibility (the weights being the squared Perron components $p_i=\varphi_i^2$): it equals $1$ when susceptibility is uniform and grows with its spread, and the bound reduces this to the single ratio $r=q_{\max}/q_{\min}$ of most- to least-susceptible group. Numerically, a susceptibility ratio of about $2.3$ (the value in the data below) permits a source--burden angle of at most $23^{\circ}$, and a ratio of $4$ at most $37^{\circ}$. Only the \emph{contrast} $r$ enters: the spread between the most and least susceptible groups caps how far the operator can separate sources from burden, and homogeneous susceptibility $(r=1)$ forces perfect alignment $(\theta=0)$. This makes the slogan of Section~\ref{sec:nonnormal} quantitative: heterogeneity is what bends the eigenvectors, and $r$ measures how much it can. The bound is one-sided by design: $Q$ \emph{can} drive $\theta$ up to $\arccos\frac{2\sqrt r}{1+r}$ but need not, since the realized angle also depends on where the Perron mass $p$ sits. If instead susceptibility and infectiousness share a common profile $(D\propto Q$ in $K=QCD)$, then $K\propto QCQ$ is symmetric and $\theta=0$; misalignment requires the infectee- and infector-side scalings to differ, of which $K=QC$ is the canonical case. The complementary source of misalignment, asymmetry of $C$ itself, is invisible to this bound and is isolated empirically in Section~\ref{sec:empirical}.
\end{remark}

\section{Non-normality and transient risk}\label{sec:transient}

Non-normality also, in principle, decouples the long run from the short run. The threshold $\rho(K)$ governs the asymptotic fate of an introduction, but the worst-case amplification over a \emph{single} generation is the largest singular value,
\begin{equation}
\sigma_{\max}(K)=\lVert K\rVert_{2}\;\ge\;\rho(K),
\end{equation}
with equality if and only if $K$ is normal. The gap $\sigma_{\max}(K)-\rho(K)$ is thus simultaneously a measure of non-normality and of transient risk; its size is the generation-time analog of what ecologists call \emph{reactivity}
\cite{neubert1997}. A caricature shows the mechanism at full strength. Take
\begin{equation}
K=\begin{pmatrix} 0.3 & 2 \\ 0 & 0.6 \end{pmatrix},
\end{equation}
read as: group~2 infects both groups, group~1 infects almost no one. Then $R_0=\rho(K)=0.6<1$, so the epidemic dies out asymptotically, yet $\sigma_{\max}(K)\approx2.11$ (a reactivity ratio of $3.5$), the eigenvectors are badly misaligned ($\kappa\approx6.7$), and a single introduction into group~2
generates the total-infection sequence $1\to2.6\to2.16\to1.48\to\cdots$: a burst that more than doubles, and concentrates in the group that barely transmits,
before subsiding.

We stress that this is a caricature. As Section~\ref{sec:empirical} shows, real contact structures are non-normal but not \emph{that} non-normal: their reactivity ratios stay near one, so the ``subcritical yet explosive'' regime is a genuine qualitative
possibility rather than a large quantitative effect. The robust empirical consequence of non-normality is not transient blow-up; it is the eigenvector misalignment of Section~\ref{sec:equity}.

\section{Evidence from 177 national contact structures}\label{sec:empirical}

To test whether these are real features of transmission systems or merely possible ones, we use the synthetic age-structured contact matrices of Prem et al.\ \cite{prem2021}, which provide, for each of $177$ countries, a $16\times16$ matrix $C$ over five-year age bands ($0$--$4,\dots,75+$). From each we build a
next-generation matrix $K=\mathrm{diag}(u)\,C$, where $u$ is the age-specific relative susceptibility to infection estimated for COVID-19 by Davies et al.\ \cite{davies2020} (lowest in children, rising through adulthood).\footnote{We use
the posterior mean consensus susceptibility of Davies et al.\ (their Extended Data
Fig.~4, estimated jointly across six countries): $0.40, 0.38, 0.79, 0.86, 0.80,
0.82, 0.88, 0.74$ for ages $0$--$9$ through $70+$, applied by ten-year band to the
corresponding five-year bands. The qualitative results are insensitive to this
choice: repeating the analysis with a uniform profile ($u\equiv1$, i.e.\ bare
contacts) and with an independent monotone gradient leaves every conclusion below
intact.} All quantities are directional and scale-free, so the overall transmissibility constant is irrelevant. We also compute each country's median age from the WPP2019 population data distributed with the contact matrices, so the demographic and contact variables share a single source.

Three findings emerge (Figure~\ref{fig:emp}). First, \emph{no matrix is normal}. The normalized departure from normality of the bare contact matrix,
$\nu=\lVert K\rVert_F^{-1}\sqrt{\lVert K\rVert_F^{2}-\sum_i|\lambda_i|^2}$, has median $0.20$ across the $177$ countries (range $0.13$--$0.38$) and is never zero; applying the Davies susceptibility profile raises it further (median $0.33$, up to $0.59$), exactly as the $K=QCD$ mechanism of Section~\ref{sec:nonnormal} predicts: the diagonal $Q$ manufactures non-normality.

Second, \emph{the source and burden eigenvectors are substantially misaligned}. The angle between the leading left eigenvector $w$ (who drives) and right eigenvector $v$ (who bears) has median $15^{\circ}$ for bare contacts and
$26^{\circ}$ for the next-generation matrix; $60$ of $177$ countries exceed $30^{\circ}$. In Italy (Figure~\ref{fig:emp}b) both distributions peak in working age, but every band above $55$ bears a larger share of infection than it contributes; for instance the $60$--$64$ band bears nearly twice the share ($0.075$) it contributes ($0.041$), so burden sits systematically older than transmission. The generation that transmits and the generation that suffers are not the same, and $\rho(K)$ reports neither. The misalignment is not an artifact of the susceptibility profile: the median angle is $15^{\circ}$ for bare contacts, $26^{\circ}$ for the Davies next-generation matrix, and $26^{\circ}$ for an independent monotone susceptibility gradient (full sweep in the supplement).

Third, \emph{the misalignment splits into a susceptibility part that the bound of Section~\ref{sec:bound} controls and a residual part contributed by contact asymmetry}. The Davies susceptibility contrast is $r=q_{\max}/q_{\min}=0.88/0.38=2.32$, for which Proposition~\ref{prop:kant} caps the source--burden angle of a \emph{symmetric}-contact operator at $\arccos\frac{2\sqrt r}{1+r}=23.4^{\circ}$. Reciprocity-symmetrizing each country's contact matrix, $\tfrac12(C+C^{\top})$, and recomputing, every one of the $177$ angles indeed falls below this bound (median $16.8^{\circ}$, maximum $20.1^{\circ}$; Figure~\ref{fig:emp}c); the gap to $23.4^{\circ}$ is expected, since the Davies profile has intermediate susceptibilities and so cannot realize the two-point extremal configuration. The full, asymmetric matrices reach $42.9^{\circ}$: the excess above the symmetric value (median $9.4^{\circ}$, up to $23.3^{\circ}$) is precisely the misalignment contributed by non-reciprocity of contact, which the bound, being a statement about symmetric $C$, does not see. Observed misalignment is thus the sum of two identifiable mechanisms: susceptibility heterogeneity, provably limited by the contrast $r$, and the asymmetry of who-contacts-whom.\footnote{Source--burden misalignment also correlates strongly with a country's median age ($r=-0.97$ for the full matrices). We do not emphasize this: median age is a summary of the same age distribution that generates $C$, and a counterfactual in which a single fixed mixing pattern is reweighted to each country's demography reproduces the correlation on its own ($r=-0.99$), while holding demography fixed and varying only mixing behavior reverses its sign ($r=+0.95$). The median-age relationship is therefore a largely mechanical consequence of age composition rather than an independent empirical regularity, and we report it only for completeness. It is not merely a developed-versus-developing contrast: the negative association persists within continental regions (within-region $r=-0.91$ in Africa, $-0.95$ in Asia, $-0.98$ in the Americas, and a weaker $-0.45$ in Europe, whose median ages span only $40$--$47$ years and so leave little to correlate), as a mechanical dependence on age composition would predict.}

Finally, a deliberate null result, and a useful one. The reactivity ratio $\sigma_{\max}(K)/\rho(K)$ stays modest for every country ($\le 1.09$ for bare contacts, $\le 1.26$ for the next-generation matrix), nowhere near the factor of $3.5$ in the caricature of Section~\ref{sec:transient}. Real contact structures are non-normal enough to relocate the burden but not enough to produce large sub-threshold transient outbreaks. This sharpens rather than weakens the paper's claim. Non-normality has two possible consequences, misalignment of the source and burden eigenvectors and transient amplification; finding the second empirically small tells us which one carries the social content. The effect worth attending to and reporting alongside $R_0$ is the eigenvector misalignment; ruling out the transient mechanism licenses that focus.

Two limitations bound these claims. The analysis is a structural, cross-sectional comparison of next-generation operators, not a dynamic outbreak simulation: it characterizes how $K$ distributes reproduction, not the realized epidemic under depletion, intervention, or behavioral feedback. And Proposition~\ref{prop:kant} assumes reciprocal (symmetric) contact; real contact matrices are only approximately reciprocal, so its bound governs the susceptibility component isolated above rather than the full observed angle, whose remainder we attribute to asymmetry.

\begin{figure}[t]
\centering
\includegraphics[width=\textwidth]{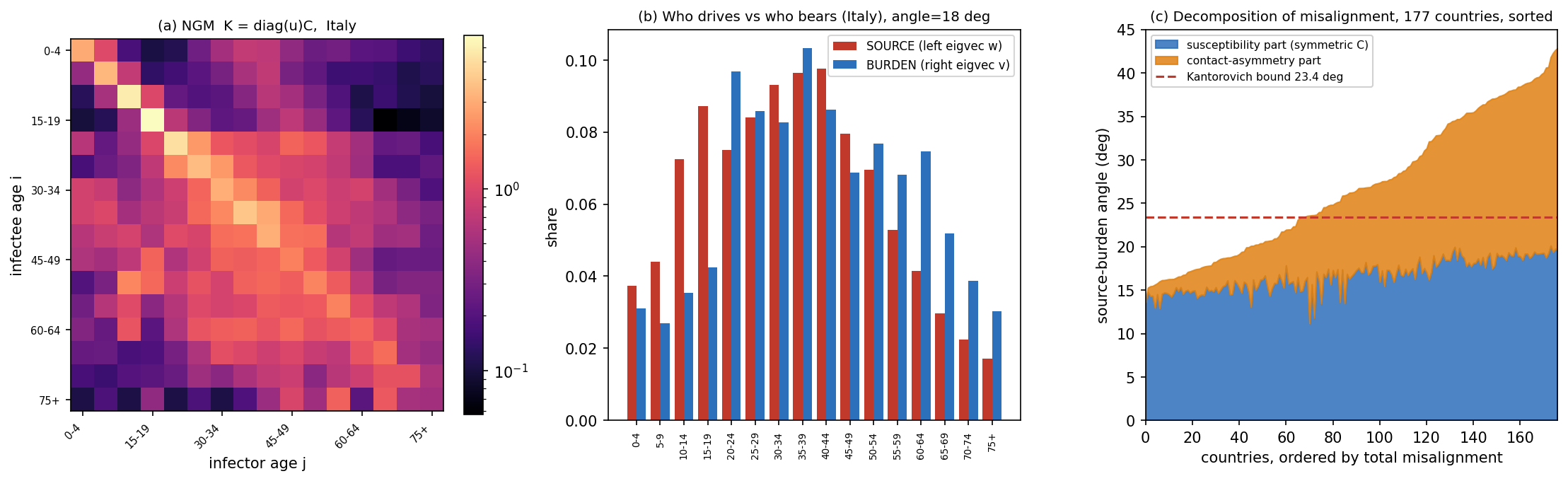}
\caption{Non-normality and eigenvector misalignment in real transmission operators. \textbf{(a)} A next-generation matrix $K=\mathrm{diag}(u)C$ for Italy (contacts from Prem et al.\ \cite{prem2021}; susceptibility $u$ from Davies et al.\ \cite{davies2020}; log color scale). \textbf{(b)} Leading left eigenvector $w$ (source: who drives transmission) versus right eigenvector $v$ (burden: who is infected), by age. The source--burden angle is the angle between these two distributions viewed as vectors: $0^{\circ}$ when they coincide (the groups that drive transmission are exactly those that bear it) and larger as they separate. Here the two sit $17.7^{\circ}$ apart, with every band
above $55$ bearing more than it transmits. \textbf{(c)} Decomposition of misalignment for all $177$ countries, ordered by total. Each country's source--burden angle splits into a susceptibility component (blue; the angle under reciprocity-symmetrized contacts) and a contact-asymmetry remainder (orange). The susceptibility component stays below the Kantorovich bound of Proposition~\ref{prop:kant} ($23.4^{\circ}$ for the Davies contrast, dashed line) in every country, as it must; the asymmetry of who-contacts-whom supplies the rest.}
\label{fig:emp}
\end{figure}

\section{Why ``how much of \texorpdfstring{$R_0$}{R0} is biology?'' is malformed}

The decomposition $K=QCD$ invites the question of how much of $R_0$ is biological ($Q$) and how much is social ($C$). The question has no answer, because the spectral radius is not multiplicative:
\begin{equation}
\rho(QCD)\neq\rho(Q)\,\rho(C)\,\rho(D)
\end{equation}
in general. $R_0$ is a nonlinear, global functional of the whole operator, in which biology and social structure are entangled through the dominant mode. Asking what fraction of it is biological is like asking what fraction of a rectangle's area is its length. The honest reading is counterfactual,
\begin{equation}
R_0(\mathcal{S})=\rho\big(K(\mathcal{S})\big),
\end{equation}
where $\mathcal{S}$ is the modeled society, its contact network, its distribution of exposure time, its behavioral baseline \cite{mossong2008,fenichel2011}. The same pathogen yields different $R_0$ in different $\mathcal{S}$ not because the
biology changed but because $\rho$ reads the whole operator, and, as Section~\ref{sec:empirical} shows, the operator is mostly social.

\section{The omission changes the decision}\label{sec:decision}

The reporting gap becomes an allocation gap the moment one intervenes, and the gap is large: across all $177$ countries the age group one should protect to minimize $R_0$ differs from the group one should protect to minimize deaths in every single one, the two sitting a median of $48$ years apart and the $R_0$-optimal choice leaving per-infection mortality about $1.8\times$ higher (quantified below). Intuitively, the two goals pull in opposite directions: cutting transmission means protecting the people who \emph{spread} infection, while cutting deaths means protecting the people who \emph{die} from it, and for an age-structured respiratory disease these are largely different groups, the young and mobile versus the old and frail. The rest of this section makes that intuition precise and measures it. Suppose a planner can reduce the susceptibility of a single age group (by vaccination, shielding, or prophylaxis), scaling $u_g\mapsto(1-\epsilon)u_g$. This perturbs $K=\operatorname{diag}(u)C$ by $\delta K=-\epsilon\,e_g e_g^{\top}K$, and the first-order eigenvalue formula $d\rho=w^{\top}\delta K\,v/(w^{\top}v)$ with $Kv=\rho v$ gives
\begin{equation}
-\frac{d\log\rho}{d\epsilon}\bigg|_{\epsilon=0}=\frac{w_g v_g}{w^{\top}v}.
\end{equation}
The reduction in $R_0$ from protecting group $g$ is thus governed by $w_g v_g$, the product of reproductive value and incidence. The mortality directly averted, by contrast, is governed by $v_g\,\mathrm{IFR}_g$: the infections borne by $g$ times their fatality risk. Both objectives weight a group's own infections $v_g$ and differ only in the second factor: $R_0$-control multiplies by the onward reproductive value $w_g$, a property of the transmission operator and precisely what the eigenvalue's sensitivity sees, while mortality-control multiplies by the fatality risk $\mathrm{IFR}_g$, a property outside it. The two agree only when reproductive value tracks severity. For respiratory infection, they run in opposite directions: reproductive value follows contact intensity and peaks in the young and working-aged, whereas fatality risk climbs monotonically with age. Minimizing $R_0$ therefore pulls toward the transmission core and minimizing mortality toward the severity tail, distinct populations, because $w$ concentrates where $v$ and $\mathrm{IFR}$ do not. The policy divergence is the source--burden misalignment of Section~\ref{sec:equity}, now sharpened by a severity weighting that compounds the older skew of the burden and recast as a choice between two interventions.

A concrete allocation shows the cost. Using Italy's contact matrix, the Davies susceptibility, an age fatality profile of the log-linear shape estimated by Levin et al.\ \cite{levin2020},\footnote{Concretely $\log_{10}\mathrm{IFR}(a)=-3.27+0.0524\,a$ with $\mathrm{IFR}$ in percent and $a$ the band-midpoint age, the log-linear metaregression of Levin et al.; these coefficients reproduce that paper's reported age anchors (e.g.\ $0.4\%$ at age $55$, $4.6\%$ at $75$, $15\%$ at $85$), and by the robustness note below they set the magnitude of the regret, not the identity of the two optima.} and a dose budget of $10\%$ of the population (all-or-nothing efficacy), two planners choose oppositely (Table~\ref{tab:alloc}). Planner~R, minimizing $\rho$, protects ages $35$--$39$ and lowers $R_0$ by $12\%$; but by displacing infection onto older groups it \emph{raises} the expected fatality of a typical infection by $13\%$. Planner~E, minimizing burden-weighted lethality, protects $75+$, cutting that lethality by $32\%$ while $\rho$ falls only $1\%$. Choosing the $R_0$-optimal target leaves per-infection mortality $1.7\times$ higher than that of the mortality-optimal target. The pattern is universal. Across all $177$ countries (budget $10\%$, full efficacy) the two optimal targets differ in every one: the $R_0$-optimal group is always working-aged ($20$--$44$, most often $25$--$29$), the mortality-optimal group is $75+$ in $176$ of $177$, and the two sit a median of $48$ years apart (range $38$--$58$). Adopting the $R_0$-optimal target leaves per-infection mortality a median $1.84\times$ higher than the mortality-optimal target would (range $1.5$--$2.0\times$), while the mortality-optimal target sacrifices almost nothing in transmission; it lowers $R_0$ by a median of $1\%$ against the $R_0$-optimal target's $11\%$. The divergence persists across budgets from $5\%$ to $20\%$ and efficacies from $50\%$ to $100\%$.

\begin{table}[h]
\centering
\small
\begin{tabular}{lcccc}
\toprule
Planner & Objective & Target group & $\Delta\rho$ & $\Delta$ per-infection lethality \\
\midrule
R & minimize $\rho(K)=R_0$ & $35$--$39$ & $-12\%$ & $+13\%$ \\
E & minimize burden $\times$ severity & $75+$ & $-1\%$ & $-32\%$ \\
\bottomrule
\end{tabular}
\caption{Two planners, one dose budget ($10\%$ of Italy's population), opposite choices. Minimizing the scalar $R_0$ selects the transmission hub and \emph{raises} the average infection fatality by shifting the burden onto the old; minimizing burden-weighted severity selects the old and barely moves $R_0$.}
\label{tab:alloc}
\end{table}

Nothing in the argument is specific to age. The sensitivity $w_g v_g$ and the bound of Proposition~\ref{prop:kant} are stated for an arbitrary diagonal susceptibility $Q$ and a general contact structure $C$, so the same divergence arises for any partition of a population into interacting groups in which those with high reproductive value are not those with high severity or cost. Age is only the most familiar such split. The same logic applies to occupation (essential and service workers sustain transmission while severe outcomes fall on their older or comorbid contacts), to geography (mobility hubs propagate infection, peripheral or under-resourced regions bear the mortality), and to socioeconomic exposure (crowded work and housing raise transmission where reduced care access lowers survival). We demonstrate only the age case, for which matched contact and fatality data exist at a global scale; the other axes would require the corresponding stratified contact matrices and severity estimates, and we flag them as the natural next application rather than as results established here.

Two caveats fix the scope. The lethality index $D=\sum_i\mathrm{IFR}_i\,v_i$ is the expected fatality of a typical infection under the invading age-distribution $v$, a compositional quantity, not a total death count; lowering $\rho$ also reduces the number eventually infected, which a dynamic model would credit as lives saved. But no size-weighting reconciles the two targets, since they sit at opposite ends of the age axis. And the fatality profile is a stylized log-linear gradient; its slope scales the magnitude of the regret, not the identity of the optima. The point is not a vaccine recommendation but a demonstration: $R_0$ cannot even express the mortality objective, because what that objective needs, the burden eigenvector $v$ weighted by severity, is exactly what $\rho(K)$ discards.

\section{What to report instead of a scalar}\label{sec:report}

If the eigenvector is the harm and the non-normality is the heterogeneity, then reporting $R_0$ alone discards exactly what a socially informed reader needs. Three quantities travel together and should be reported together:
\begin{enumerate}
\item $\rho(K)=R_0$, \emph{whether} infection invades (the threshold);
\item the leading eigenvectors $v$ and $w$, or their misalignment $\kappa$, \emph{to whom} it happens and \emph{who} drives it (the equity content, which Section~\ref{sec:empirical} shows is large in practice);
\item the reactivity $\sigma_{\max}(K)$, \emph{how violently} a badly placed introduction can surge before the threshold takes hold (a transient risk that is small for typical contact structures but cheap to check and occasionally decisive).
\end{enumerate}
None of these is a refinement of $R_0$; each is a distinct piece of the matrix that the spectral radius removes. All three are computed en route to $\rho(K)$, so the triple costs nothing, and it converts $R_0$ from a number that conceals its own assumptions into one that declares them. The added value is largest precisely where contact is strongly structured and susceptibility differs markedly across groups, the regime in which $\rho(K)$ and its eigenvectors diverge; for a population near homogeneity, the eigenvectors nearly coincide, and $\rho(K)$ alone loses little.

\section{Conclusion}

$R_0=\rho(K)$ is an act of compression, and like any compression it is defined as much by what it deletes as by what it keeps. It keeps the dominant eigenvalue: the yes-or-no of invasion. It deletes the eigenvectors, which say where the infection concentrates, and the non-normality, which says how far the short run can stray from the long run. On real national contact structures, those deletions are not incidental: every one of $177$ transmission operators is non-normal, and its source and burden directions diverge by tens of degrees, a divergence that susceptibility heterogeneity provably bounds and contact asymmetry pushes further. The generation that drives an epidemic is systematically not the one that suffers it, and the scalar threshold is silent on the difference.

We do not suggest replacing $R_0$, nor do we claim that its eigenvectors always matter. The practical recommendation is narrower: in populations with substantial age or social heterogeneity and differential susceptibility, the regime in which the source and burden directions measurably diverge, reporting the leading eigenvectors and the reactivity alongside $\rho(K)$ carries information the scalar cannot, at no extra computational cost. Where a population is close to homogeneous, the eigenvectors nearly coincide, and $\rho(K)$ largely suffices; the case for the fuller summary grows with the heterogeneity it is meant to expose.

The eigenvalue is the pathogen's verdict on the population; the eigenvectors are the population itself. Where that population is strongly structured, to report only the first is to describe an epidemic while omitting whose epidemic it is.

\end{document}